\documentclass[12pt, reqno]{amsart}
\usepackage{times,amsmath,amssymb,amsthm,amscd, mathrsfs, color}
\usepackage[all]{xy}

\newtheorem{thm}{Theorem}[section]
\newtheorem{cor}[thm]{Corollary}
\newtheorem{lem}[thm]{Lemma}
\newtheorem{pro}[thm]{Proposition}

\theoremstyle{definition}

\numberwithin{equation}{section}

\newcommand{\R}{{\mathbb R}}

\newcommand{\D}{{\mathbb D}}

\begin{document}

\title[]
{Compact differences of composition operators on large weighted Bergman spaces}

\thanks{The author was supported by the NRF (No. 2018R1D1A1B07046890) of Korea}
\subjclass[2010]{}

\author[I. Park]{Inyoung Park}
\address{Department of Mathematics Korea University, Seoul, 02841, Republic of Korea}
\email{iypark26@gmail.com}

\begin{abstract}
While there have been extensive studies regarding the theory of composition operators in standard Bergman spaces, there have not been many results pertaining to large Bergman spaces due to a lack of useful tools. In this paper, we give the characterizations of the compact differences of composition operators in Bergman spaces with the exponential type weight using a newly defined Riemannian distance. Furthermore, we give a sufficient condition for the question when two composition operators lie in the same component.
\end{abstract}

\maketitle

\section{Introduction}
Let $\phi$ be a holomorphic self-map of the unit disk $\D$. The composition operator $C_\phi:H(\D)\rightarrow H(\D)$ is defined by $C_\phi f:=f\circ\phi$, where $H(\D)$ is the space of holomorphic functions on $\D$. For the integrable radial function $\omega$, let $L^p(\omega dA)$ be the space of all measurable functions $f$ on $\D$ such that

\begin{align*}
\|f\|^p_{p}:=\int_\D |f(z)|^p\omega(z)dA(z)<\infty,\quad0<p<\infty,
\end{align*}

where $dA(z)$ is the normalized area measure on $\D$. We denote $A^p(\omega)=L^p(\omega dA)\cap H(\D)$ and we use the notation $A^p_\alpha(\D)$ when $\omega(z)=(1-|z|)^\alpha$, $\alpha>-1$. Throughout this paper, we consider radial weights of the form $\omega(r)=e^{-\varphi(r)}$ where $\varphi(r)=-A(r)\log(1-r)$ with $A(r)$ is non-decreasing and $A(r)\rightarrow\infty$ as $r\rightarrow1^-$ and $A(0)\neq0$. We let $\varphi\in\mathcal{C}^2$,
\begin{align}\label{taudef}
(\Delta\varphi(z))^{-\frac{1}{2}}\asymp\tau(z),
\end{align}
and we assume that $\tau(r)$ and $-\tau'(r)$ decrease to $0$ near $1$. Furthermore, we use the peak functions given by \cite{BDK} under the additional following conditions to obtain a necessary condition of our main theorems:
\begin{align}\label{addcondit}
\lim_{r\rightarrow1^-}\tau'(r)\log\frac{1}{\tau(r)}=0\quad\text{or}\quad\lim_{r\rightarrow1^-}\frac{\tau(r)}{(1-r)^{A}}=\infty
\end{align}
for $A>0$. Now, we say that the weight $\omega$ belongs to the class $\mathcal{W}$ if it satisfies all the conditions above. Concerning the standard weight $\varphi(z)=-\alpha\log(1-|z|)$, $\alpha>-1$, we can easily check that $\tau'(r)\nrightarrow0$ as $r\rightarrow1^-$, thus the weight class $\mathcal{W}$ does not contain standard weights and is composed of regular fast weights. A typical weight example in the class $\mathcal{W}$ is $\omega(z)=e^{-\frac{1}{1-|z|}}$, and the reader can refer to \cite{PP} for other examples belonging to $\mathcal{W}$.\\
\indent It is well known that all composition operators are bounded in standard Bergman spaces $A^p_\alpha(\D)$ owing to the Littlewood subordination principle. The compactness characterization of $C_\phi$ on $A^p_\alpha(\D)$ is also well known by the non-existence of angular derivatives of the inducing function $\phi$ of \cite{MS}. On the other hand, Kriete and MacCluer showed that not every composition operator is bounded in Bergman space with $\omega\in\mathcal{W}$ in \cite{KM}. They showed that if
\begin{align}\label{suffibddwithoutU}
\limsup_{|z|\rightarrow1}\frac{\omega(z)}{\omega(\phi(z))}<\infty,
\end{align}
then $C_\phi$ is bounded by $A^2(\omega)$. Furthermore, \cite{KM} gives a characterization with respect to the angular derivative (see Section 2.4 for a definition) as follows: $\phi$ induces an unbounded composition operator $C_\phi$ on $A^2(\omega)$ if there is a boundary point, $\zeta$, such that $|\phi'(\zeta)|<1$. Recently, we have shown that (\ref{suffibddwithoutU}) is an equivalent condition for the boundedness of $C_\phi$ on $A^p(\omega)$, where $0<p<\infty$ in \cite{P} when $\omega$ belongs to the class $\mathcal{W}$. Thus, the boundedness of $C_\phi$ on $A^2(\omega)$ implies the boundedness of $C_\phi$ on $A^p(\omega)$ for the full range of $p$. In the following theorem, they gave equivalent conditions for the compactness of composition operators on $A^2(\omega)$.
\begin{thm}\cite{KM}\label{knownresult2}
Let $\omega$ belong to the class $\mathcal{W}$. The following conditions are equivalent:
\begin{itemize}\label{angulcompact}
\item[(1)]$C_\phi$ is compact on $A^2(\omega)$.
\item[(2)]$\lim_{|z|\rightarrow1}\frac{\omega(z)}{\omega(\phi(z))}=0$.
\item[(3)]$|\phi'(\zeta)|>1$ for all $\zeta$ in $\partial\D$.
\end{itemize}
\end{thm}
 The purpose of this paper is to study when the difference of two composition operators is compact on $A^p(\omega)$, $0<p<\infty$. This problem was originally raised from the question of the component in the space of composition operators on $H^2$ under the operator norm topology in \cite{SS}. \cite{SS} suggested the following conjecture; The set of all composition operators that differ from a compact operator forms a component of the bounded composition operator space $\mathcal{C}(H^2)$. Although this turned out to be false, the study of compact differences has continued regarding various holomorphic function spaces since it can give a partial answer to their question. Afterwards, \cite{MOZ} showed that the compactness of $C_\phi-C_\psi$ acting on $H^\infty$ is characterized in terms of the pseudo-hyperbolic distance between $\phi$ and $\psi$, i.e., $\rho(\phi,\psi)=\left|\frac{\phi-\psi}{1-\phi\overline{\psi}}\right|$. Using the pseudo-hyperbolic metric, Moorhouse characterized the compact differences of composition operators acting on $A^2_\alpha(\D)$ by the angular derivative cancellation property of \cite{M} as follows; $C_\phi-C_\psi$ is compact on $A^2_\alpha(\D)$ if and only if

\begin{align}\label{standardcondi}
\lim_{|z|\rightarrow1}\rho(\phi(z),\psi(z))\left(\frac{1-|z|^2}{1-|\phi(z)|^2}+\frac{1-|z|^2}{1-|\psi(z)|^2}\right)=0.
\end{align}

In this paper, we show that the angular derivative cancellation property also occurs in the case of large Bergman spaces. More precisely, even though either $|\phi'(\zeta)|=1$ or $|\psi'(\zeta)|=1$, $C_\phi-C_\psi$ can be compact on $A^p(\omega)$ when the distance between their images of $\phi$ and $\psi$ are sufficiently close to the boundary point $\zeta$. In \cite{M}, we found that the pseudo-hyperbolic distance plays a key role in Bergman spaces with weights of the form $\omega(z)=(1-|z|)^\alpha$, $\alpha>-1$, but it is not applicable when using weights of the form $\omega(z)=e^{-\varphi(z)}$. Instead, we use the Riemannian distance induced by the metric $\frac{1}{\tau(z)^2}(dx^2+dy^2)$: For $z,w\in\D$,

\begin{align*}
d_\tau(z,w):=\inf_\gamma\int^1_0\frac{|\gamma'(t)|}{\tau(\gamma(t))}dt,
\end{align*}
where the infimum is taken over all piecewise smooth curves $\gamma$ connecting $z$ and $w$. Moreover, to assure the boundedness of composition operators in $\mathcal{C}(A^p(\omega))$, we need another distance induced by the metric $\varphi'(|z|)^2(dx^2+dy^2)$: For $z,w\in\D$,

\begin{align}\label{verphi}
d_\varphi(z,w):=\inf_\gamma\int^1_0\varphi'(|\gamma'(t)|)|\gamma'(t)|dt.
\end{align}
We will study for details in Proposition \ref{impot}. From (\ref{lowerbd}) of Lemma \ref{key-1}, we see that the distance $d_\varphi$ is also complete since we have the following relation:
\begin{align*}
d_\varphi(z,w)\gtrsim d_\tau(z,w)\gtrsim d_h(z,w)
\end{align*}
where $d_h(z,w)$ is the hyperbolic distance associated with the Bergman metric.
Now, we state our first main result:
\begin{thm}\label{m-thm1}
Let $\omega$ belong to the class $\mathcal{W}$ and $C_\phi$, $C_\psi$ be bounded on $A^p(\omega)$ for some $p\in(0,\infty)$. If there is $R>0$ such that $d_\varphi(\phi(z),\psi(z))<R$ for all $z\in\D$ and
\begin{align}\label{equiv}
\lim_{|z|\rightarrow1^-}\rho_\tau(\phi(z),\psi(z))\left(\frac{\omega(z)}{\omega(\phi(z))}+\frac{\omega(z)}{\omega(\psi(z))}\right)=0,
\end{align}
where $\rho_\tau(z,w)=1-e^{-d_\tau(z,w)}$, then $C_\phi-C_\psi$ is compact on $A^p(\omega)$. Conversely, if $C_\phi-C_\psi$ is compact on $A^p(\omega)$ then (\ref{equiv}) holds.
\end{thm}

While our approach is based on Moorhouse's method, extra work needs to be completed to deal with the exponential-weights case. Here, we note that our all results are still valid in the case of $A^p_\alpha(\D)$ since
\begin{align*}
\rho_\tau(z,w)\approx\rho(z,w).
\end{align*}
In addition, we give conditions which a single composition operator can be expressed by a finite sum of composition operators modulo compact operators.
\begin{thm}
Let  $\omega\in\mathcal{W}$, and $C_\phi, C_{\phi_1},\ldots, C_{\phi_M}$ be bounded operators on $A^p(\omega)$ for some $p\in(0,\infty)$. We define a set as
\begin{align*}
F(\phi)=\left\{\zeta\in\partial\D:\limsup_{z\to \zeta}\frac{\tau(z)}{\tau(\phi(z))}>0\right\}.
\end{align*}
Suppose $F(\phi_i)\cap F(\phi_j)=\emptyset$ for $i\neq j$ and $F(\phi)=\bigcup^M_{j=1}F(\phi_j)$. If there is $R>0$ such that $d_\varphi(\phi(z),\phi_j(z))<R$ for all $z\in\D$ and each $j=1,\dots,M$  and
\begin{align}\label{equiv1}
\lim_{z\rightarrow\zeta}\rho_\tau(\phi(z),\phi_j(z))\left(\frac{\omega(z)}{\omega(\phi(z))}+\frac{\omega(z)}{\omega(\phi_j(z))}\right)=0,\quad\forall\zeta\in F(\phi_j),
\end{align}
then $C_\phi-\sum^M_{j=1}C_{\phi_j}$ is compact on $A^p(\omega)$ for all $p\in(0,\infty)$. Conversely, if $C_\phi-\sum^M_{j=1}C_{\phi_j}$ is compact on $A^p(\omega)$, then (\ref{equiv1}) holds.
\end{thm}

In Section 5, we give a partial answer for the question regarding the component structure of $\mathcal{C}(A^p(\omega))$ under the operator norm topology. Finally, we close this section by mentioning that the distance $\rho_\tau(z,w)$ in all our theorems can be replaced by the following explicit function:

\begin{align*}
f(z,w)=1-\exp\left(-\frac{|z-w|}{\min(\tau(z),\tau(w))}\right).
\end{align*}

\bigskip

\textit{Constants.}
 In the rest of the paper, we use the notation
$X\lesssim Y$  or $Y\gtrsim X$ for nonnegative quantities $X$ and $Y$ to mean
$X\le CY$ for some inessential constant $C>0$. Similarly, we use the notation $X\approx Y$ if both $X\lesssim Y$ and $Y \lesssim X$ hold.

\section{Preliminary}
We assumed that $\tau'(r)\rightarrow0$ when $r\rightarrow1^-$ in Section 1. As such, there exist constants $c_1, c_2>0$ such that $\tau(z)\leq c_1(1-|z|)$ for $z\in\D$ and

\begin{align}\label{comparable}
|\tau(z)-\tau(w)|\leq c_2|z-w|,\quad\text{for}\quad z,w\in\D.
\end{align}

We let $D(z,\delta\tau(z))$ be a Euclidean disk centred at $z$ with radius $\delta\tau(z)$, and we use the notation $D(\delta\tau(z)):=D(z,\delta\tau(z))$ for simplicity. Throughout this paper, we denote

\begin{align*}
m_\tau:=\frac{\min(1,c_1^{-1},c_2^{-1})}{4}.
\end{align*}

Using (\ref{comparable}) and the definition of $m_\tau$, we obtain that for $0<\delta\leq m_\tau$,

\begin{align}\label{equiquan}
\frac{3}{4}\tau(z)\leq\tau(w)\leq\frac{5}{4}\tau(z)\quad\text{if}\quad w\in D(\delta\tau(z)).
\end{align}

We can also refer to the proof of lemma 2.1 of \cite{PP} for the inequality above. If the Borel function $\tau:\D\rightarrow(0,+\infty)$ satisfies the condition (\ref{equiquan}), we call it a radius function.

\subsection{Radius functions and associated distances}

Using the radius function $\tau$, we can define the following set

\begin{align*}
B_\tau(z,r):=\{w\in\D:d_\tau(z,w)<r\},
\end{align*}

where

\begin{align*}
d_\tau(z,w)=\inf_\gamma\int_0^1\frac{|\gamma'(t)|}{\tau(\gamma(t))}\,dt
\end{align*}

is taken piecewise over $\mathcal{C}^1$ curves $\gamma:[0,1]\rightarrow\D$ that connect $z$ and $w$. The following proposition gives the inclusion relation between $B_\varphi(z,r)$ and the Euclidian disks with center $z$. We use the argument of the proof of proposition 5 of \cite{D}.

\begin{pro}\label{setinclus}
Let $\omega\in\mathcal{W}$ and $0<\delta<m_\tau$. Then, we have
\begin{align*}
B_\tau(z,r)\subset D\left(z,2r\tau(z)\right)\subset B_\tau(z,4r)\quad\text{for}\quad0<r<\frac{\delta}{2}.
\end{align*}
\end{pro}

\begin{proof}
For a given $z\in\D$ and any point $w$ in $B_\tau(z,r)$, there is $s>0$ that satisfies
\begin{align}\label{distzw}
|z-w|=s\delta\tau(z).
\end{align}
Now, we consider an arbitrary curve $\gamma$ that connects $z$ and $w$ with $\gamma(0)=z$ and $\gamma(1)=w$. We can get the minimum value $0<t_0\leq1$ such that
\begin{align*}
|z-\gamma(t_0)|=\min\{s,1\}\delta\tau(z).
\end{align*}
By (\ref{equiquan}), we have $|z-\gamma(t_0)|\leq\delta\tau(z)$ so that
\begin{align*}
d_\tau(z,w)\geq\int^{t_0}_0\frac{|\gamma'(t)|}{\tau(\gamma(t))}dt\geq\frac{1}{2\tau(z)}
\int^{t_0}_0|\gamma'(t)|dt\geq\min\{s,1\}\frac{\delta}{2}.
\end{align*}
Since $\frac{\delta}{2}>r>d_\tau(z,w)\geq\min\{s,1\}\frac{\delta}{2}$, the inequality above is bounded below by
\begin{align}\label{distlowerbd}
r>d_\tau(z,w)\geq s\frac{\delta}{2}=\frac{|z-w|}{2\tau(z)}.
\end{align}
Thus, we have proved the first inclusion. For the second inclusion, if we take $\gamma(t)=(1-t)z+tw$ for $w\in D(z,2r\tau(z))$, (\ref{equiquan}) gives
\begin{align*}
d_\tau(z,w)=\inf_\gamma\int^{1}_0\frac{|\gamma'(t)|}{\tau(\gamma(t))}dt\leq 2\frac{|z-w|}{\tau(z)}<4r.
\end{align*}
Hence, the proof is complete.
\end{proof}

\begin{lem}\label{key-1}
Let $\omega\in\mathcal{W}$. Then, there exists $0<r_0<1$ such that $\varphi'(r)\geq\frac{1}{2\tau(r)}$ for $r>r_0$. Moreover, $\lim_{r\rightarrow1^-}\left(\frac{1}{\varphi'(r)}\right)'=0$.
\end{lem}
\begin{proof}
Since $-\tau'(|z|)$ tends to $0$ as $|z|\rightarrow1$, there is $0<r<1$ such that
\begin{align*}
-\tau'(|z|)\leq r\quad\text{for}\quad|z|\in[r,1).
\end{align*}
 Thus, we can choose $r<r_0<1$ that satisfies $2\tau(r_0)<\tau(r)$ so that for $r_0<|z|<1$,
\begin{align}\label{lowerbd}
\tau(z)\varphi'(|z|)=\frac{\tau(z)}{|z|}\int^{|z|}_0s\Delta\varphi(s)ds&\geq\frac{\tau(z)}{|z|}
\int^{|z|}_{r}\frac{-\tau'(s)}{\tau(s)^2}ds\nonumber\\&\geq\frac{\tau(z)}{|z|}
\left(\frac{1}{\tau(z)}-\frac{1}{\tau(r)}\right)\nonumber\\&\geq\frac{1}{|z|}\left(1-\frac{\tau(r_0)}{\tau(r)}\right)>\frac{1}{2}.
\end{align}
Therefore $\liminf_{r\rightarrow1^-}\tau(r)\varphi'(r)\neq0$. Assume that there is a sequence such that $\lim_{n\rightarrow\infty}\tau(r_n)\varphi'(r_n)=a>\frac{1}{2}$. By the L'H\^{o}spital's rule,
\begin{align*}
\lim_{r\rightarrow1}\frac{\tau(r)}{1/\varphi'(r)}=\lim_{r\rightarrow1}\frac{\tau'(r)}{(1/\varphi'(r))'}\neq0.
\end{align*}
Thus, $\lim_{r\rightarrow1^-}\left(\frac{1}{\varphi'(r)}\right)'$ must be zero since $\tau'(r)\rightarrow0$ and
\begin{align*}
\varphi'(r)=-A'(r)\log(1-r)+A(r)\frac{1}{1-r}
\end{align*}
is increasing as $r\rightarrow1^-$. For the case $\lim_{r\rightarrow1^-}\tau(r)\varphi'(r)=\infty$, we easily obtain $\lim_{r\rightarrow1^-}\frac{\varphi''(r)}{\varphi'(r)^2}=0$ using the formula $\Delta\varphi(z)=\varphi''(r)+\frac{1}{r}\varphi'(r)$.
\end{proof}
 Recently, in Theorem 3.3 of \cite{HLS}, the authors gave the following inequality:
\begin{align}\label{happy}
e^{-d_\tau(z,w)}\leq C(M)\left(\frac{\min(\tau(z),\tau(w))}{|z-w|}\right)^M,\quad z,w\in\D,
\end{align}
when $\omega\in\mathcal{W}$. Using the inequality (\ref{happy}) together with Lemma \ref{key-1}, we can obtain the following relation.
\begin{lem}\label{setinclus2}
Let $\omega\in\mathcal{W}$. Define the following set
\begin{align*}
B_\varphi(z,R):=\{w\in\D:d_\varphi(z,w)<R\},
\end{align*}
where $d_\varphi(z,w)$ is defined in (\ref{verphi}). Then there is $R'>R>0$ such that
\begin{align*}
B_\varphi(z,R)\subset D(z,R'/\varphi'(|z|)).
\end{align*}
\end{lem}
\begin{proof}
Since we have $\left(\frac{1}{\varphi'(r)}\right)'\rightarrow0$ as $r\rightarrow0$ by Lemma \ref{key-1}, for any $\epsilon>0$, there exists a compact set $\mathcal{K}$ of $\D$ such that
\begin{align*}
\left|\frac{1}{\varphi'(|z|)}-\frac{1}{\varphi'(|w|)}\right|<\epsilon|z-w|,
\end{align*}
where $z,w\in\D\setminus \mathcal{K}$. Then we can replace $\tau$ in (\ref{happy}) with $1/\varphi'$, thus we have
\begin{align}\label{need}
\frac{|z-w|}{\min(1/\varphi'(|z|),1/\varphi'(|w|))}\leq e^{\frac{1}{M}d_\varphi(z,w)+\frac{C}{M}}.
\end{align}
Therefore, we have $R'>e^{\frac{R}{M}+\frac{C}{M}}$ satisfying the inclusion.
\end{proof}

\begin{pro}\label{impot}
Let $\omega\in\mathcal{W}$ and $d_\varphi(z,w)<R$. Then we have
\begin{align*}
e^{-\varphi(z)}\approx e^{-\varphi(z_s)},\quad\text{where}\quad z_s=(1-s)z+sw,
\end{align*}
for all $0\leq s\leq1$.
\end{pro}
\begin{proof}
From (\ref{need}) in Lemma \ref{setinclus2}, there is $R'>0$ such that
\begin{align}\label{confusing}
|z-z_s|\leq|z-w|<\frac{R'}{\max(\varphi'(|z|),\varphi'(|w|))},\quad\forall s\in[0,1]
\end{align}
for $d_\varphi(z,w)<R$. Moreover, since $|\varphi'(z)|=\frac{1}{2}\varphi'(|z|)$ we have
\begin{align*}
d_\varphi(z,z_s)=\int^1_0|\gamma'(t)|\varphi'(|\gamma(t)|)dt&=2\int^1_0|\gamma'(t)||\varphi'(\gamma(t))|dt\\
&\geq\left|\int^1_0\gamma'(t)\varphi'(\gamma(t))dt\right|=|\varphi(z)-\varphi(z_s)|.
\end{align*}
Since $|(1-s)z+sw|\leq\max(|z|,|w|)$ and $\varphi'(r)$ is increasing, (\ref{confusing}) gives
\begin{align*}
|\varphi(z)-\varphi(z_s)|\leq d_\varphi(z,z_s)&\leq|z-z_s|\max(\varphi'(|z|),\varphi'(|z_s|))\\&\leq|z-w|\max(\varphi'(|z|),\varphi'(|w|))<R'.
\end{align*}
Thus, we complete our proof.
\end{proof}
Therefore, if $C_\phi$ and $C_\psi$ are bounded operators and $d_\varphi(\phi(z),\psi(z))<R$ for some $R>0$, we can tell the composition operator $C_{\phi_s}$ is also bounded on $A^p(\omega)$ where $\phi_s(z)=(1-s)\phi(z)+s\psi(z)$, $0\leq s\leq1$ by (\ref{suffibddwithoutU}).

\subsection{The size estimate of the growth of $f$ in $A^p(\omega)$}

\begin{lem}\cite[Lemma 2.2]{PP}\label{subharmoni-1}
Let $\omega=e^{-\varphi}$, where $\varphi$ is a subharmonic function. Suppose the function $\tau$ satisfies properties (\ref{equiquan}) and  $\tau(z)^2\Delta\varphi(z)\lesssim1$. For $\beta\in\R$ and $0<p<\infty$, there exists a constant $M\geq1$ such that
\begin{align*}
|f(z)|^p\omega(z)^\beta\leq\frac{M}{\delta^2\tau(z)^2}\int_{D(\delta\tau(z))}|f|^p\omega^\beta dA
\end{align*}
for a sufficiently small $\delta>0$, and $f\in H(\D)$.
\end{lem}
In order to estimate the difference of function values at two different points, we need to estimate the growth of the functions in $A^p(\omega)$. The following inequality shows the size of the differential of functions in  $A^p(\omega)$. In fact, its proof is the same with the case of a doubling measure $\Delta\varphi$ which is found in lemma 19 of \cite{MMO}. Moreover, we can find the same inequality in our setting in Lemma 3.3 of \cite{HLS}.
\begin{lem}\label{derivsubmean}
Let $\omega=e^{-\varphi}$, where $\varphi$ is a subharmonic function. Suppose the function $\tau$ satisfies properties (\ref{equiquan}) and  $\tau(z)^2\Delta\varphi(z)\lesssim1$. Then
\begin{align*}
|f'(z)|^pe^{-\varphi(z)}\lesssim\frac{1}{\tau(z)^{2+p}}\int_{D(\delta\tau(z))}|f(\xi)|^p e^{-\varphi(\xi)}\,dA(\xi)
\end{align*}
for a sufficiently small $0<\delta<m_\tau$ and $f\in H(\D)$.
\end{lem}
Using Lemma \ref{derivsubmean}, we obtain the following Proposition which plays a crucial role in our proof.

\begin{lem}\label{inequality}
Let $\omega\in\mathcal{W}$, $0<p<\infty$ and $R>0$. For $w\in\D$ satisfying $d_\varphi(z,w)<R$, we have
\begin{align*}
|f(z)-f(w)|^pe^{-\varphi(z)}\lesssim\frac{\rho_\tau(z,w)^p}{\tau(z)^2}\int_{D(\delta\tau(z))}|f(\xi)|^p e^{-\varphi(\xi)}\,dA(\xi),
\end{align*}
when $|z-w|\leq\delta/2\tau(z)$ where $0<\delta<m_\tau$ and $f\in H(\D)$.
\end{lem}

\begin{proof}
Given a fixed $z\in\D$ and any point $w$ in $D(z,\delta/2\tau(z))$, there is $z_s=(1-s)z+sw$ such that
\begin{align*}
|f(z)-f(w)|=|f'(z_s)(z-w)|
\end{align*}
by the mean value theorem. From Proposition \ref{impot}, Lemma \ref{derivsubmean} and the equality above, we have
\begin{align}\label{meanvalue}
|f(z)-f(w)|^pe^{-\varphi(z)}&\lesssim|f'(z_s)|^p|z-w|^pe^{-\varphi(z_s)}\nonumber\\
&\lesssim\frac{|z-w|^p}{\tau(z_s)^{p+2}}\int_{D(\delta/4\tau(z_s))}|f(\xi)|^p e^{-\varphi(\xi)}\,dA(\xi).
\end{align}
Since $||\xi-z|-s|z-w||\leq|\xi-z_s|<\delta/4\tau(z_s)$ for $\xi\in D(\delta/4\tau(z_s))$,
\begin{align*}
|z-\xi|\leq\frac{\delta}{4}\tau(z_s)+|z-w|\leq \frac{\delta}{2}\tau(z)+\frac{\delta}{2}\tau(z)\leq\delta\tau(z).
\end{align*}
Thus, we have
\begin{align}\label{set}
D(\delta/4\tau(z_s))\subset D(\delta\tau(z)).
\end{align}
Using (\ref{set}) and (\ref{distlowerbd}) of Proposition \ref{setinclus}, we obtain the following inequality from (\ref{meanvalue})
\begin{align*}
|f(z)-f(w)|^pe^{-\varphi(z)}&\lesssim\frac{|z-w|^p}{\tau(z)^{p+2}}\int_{D(\delta\tau(z))}|f(\xi)|^p e^{-\varphi(\xi)}\,dA(\xi)\\
&\lesssim\frac{d_\tau(z,w)^p}{\tau(z)^2}\int_{D(\delta\tau(z))}|f(\xi)|^p e^{-\varphi(\xi)}\,dA(\xi).
\end{align*}
Since $|z-w|\leq\delta/2\tau(z)$, we have $d_\tau(z,w)<\delta$ by Proposition \ref{setinclus}. Thus, by the relation $x\leq e^\delta(1-e^{-x})$ for a positive $x$ near to $0$ we obtain
\begin{align*}
|f(z)-f(w)|^pe^{-\varphi(z)}\lesssim\frac{\rho_\tau(z,w)^p}{\tau(z)^2}\int_{D(\delta\tau(z))}|f(\xi)|^p e^{-\varphi(\xi)}\,dA(\xi).
\end{align*}
\end{proof}

\subsection{Carleson measure theorem}
A positive Borel measure $\mu$ in $\D$ is called a (vanishing) Carleson measure for $A^p(\omega)$ if the embedded $A^p(\omega)\subset L^p(\omega d\mu)$ is (compact) continuous where

\begin{align}\label{embedding}
L^p(\omega d\mu):=\left\{f\in\mathcal{M}(\D)\big|\int_\D|f(z)|^p\omega(z)d\mu(z)<\infty\right\},
\end{align}

and $\mathcal{M}(\D)$ is a set of $\mu$-measurable functions on $\D$. Now, we introduce Carleson measure theorem on $A^p(\omega)$, as given by \cite{PP}.

\begin{thm}[Carleson measure theorem]
Let $\omega\in\mathcal{W}$ and $\mu$ be a positive Borel measure on $\D$. Then, for $0<p<\infty$, we have
\begin{itemize}
\item[(1)]The embedded $I:A^p(\omega)\rightarrow L^p(\omega d\mu)$ is bounded if and only if, for a sufficiently small $\delta\in(0,m_\tau)$, we have
\begin{align*}
\sup_{z\in\D}\frac{\mu(D(\delta\tau(z)))}{\tau(z)^{2}}<\infty.
\end{align*}
Here, we call the set $D(\delta\tau(z))$ a Carleson box and $\|I\|\approx\sup_{z\in\D}\frac{\mu(D(\delta\tau(z)))}{\tau(z)^{2}}$.
\item[(2)]The embedded $I:A^p(\omega)\rightarrow L^p(\omega d\mu)$ is compact if and only if, for a sufficiently small $\delta\in(0,m_\tau)$, we have
\begin{align*}
\lim_{|z|\rightarrow1}\frac{\mu(D(\delta\tau(z)))}{\tau(z)^{2}}=0.
\end{align*}
\end{itemize}
\end{thm}

By the measure theoretic change of variables, we have

\begin{align}\label{measurechange}
\|C_\phi f\|_{p}^p=\int_\D|f\circ\phi|^p\omega dA=\int_\D|f|^p\omega\omega^{-1}[\omega dA]\circ\phi^{-1},
\end{align}

where

\begin{align*}
\omega^{-1}[\omega dA]\circ\phi^{-1}(E)=\int_{\phi^{-1}(E)}\frac{\omega(z)}{\omega(\phi(z))}dA(z)
\end{align*}

for any measurable subsets $E$ of $\D$. Therefore, Carleson measure theorem states that $\omega^{-1}[\omega dA]\circ\phi^{-1}$ is a Carleson measure if and only if $C_\phi$ is bounded by $A^p(\omega)$ for all $p\in(0,\infty)$. Moreover, the operator norm of $C_\phi$ is obtained by

\begin{align}\label{opnorm}
\|C_\phi\|\approx\sup_{z\in\D}\frac{\omega^{-1}[\omega dA]\circ\phi^{-1}(D(\delta\tau(z)))}{\tau(z)^{2}}.
\end{align}

\subsection{Angular derivative}
For a boundary point, $\zeta$, and $\alpha>1$, we define the nontangential approach region at $\zeta$ by

\begin{align*}
\Gamma(\zeta,\alpha)=\{z\in\D:|z-\zeta|<\alpha(1-|z|)\}.
\end{align*}

A function $f$ is said to have a nontangential limit at $\zeta$ if

\begin{align*}
\angle\lim_{\substack{z\to \zeta\\ z\in \Gamma(\zeta,\alpha)}}f(z)<\infty\quad\text{for each}\quad\alpha>1.
\end{align*}

We say $\phi$ has a finite angular derivative at a boundary point, $\zeta$, if there is $\eta$ on the circle such that

\begin{align*}
\phi'(\zeta):=\angle\lim_{\substack{z\to \zeta\\ z\in \Gamma(\zeta,\alpha)}}\frac{\phi(z)-\eta}{z-\zeta}<\infty,\quad\text{for each}\quad\alpha>1.
\end{align*}
By the Julia-Caratheodory Theorem, it is well known that
\begin{align*}
\beta_\phi(\zeta)=\liminf_{z\rightarrow\zeta}\frac{1-|\phi(z)|}{1-|z|}<\infty
\end{align*}
is an equivalent condition for the finite angular derivative. If $\beta_\phi(\zeta)<\infty$, then $\phi(E(\zeta,k))\subseteq E(\phi(\zeta),k\beta_\phi(\zeta))$ for every $k>0$, where

\begin{align*}
E(\zeta,k)=\{z\in\D:|\zeta-z|^2\leq k(1-|z|^2)\}.
\end{align*}

So, when $\beta_\phi(\zeta)\leq1$ and $k=\frac{1-r}{1+r}$, we have

\begin{align}\label{d<=1}
\phi\left(E\left(\zeta,\frac{1-r}{1+r}\right)\right)\subseteq E\left(\phi(\zeta),\frac{1-r}{1+r}\right),\quad r\in(0,1).
\end{align}

A computation shows that (\ref{d<=1}) gives $\left|\frac{1+r}{2}\phi(\zeta)-\phi(r\zeta)\right|\leq \frac{1-r}{2}$, thus $|\phi(r\zeta)|\geq r$ for $0<r<1$. The readers can refer to \cite{CM} for details above. The next result follows immediately.

\begin{lem}\label{d>1}
Let $\omega$ be a fast weight. Then, $\beta_\phi(\zeta)>1$ for $\zeta\in\partial\D$ if
\begin{align*}
\lim_{z\rightarrow\zeta}\frac{\omega(z)}{\omega(\phi(z))}=0.
\end{align*}
\end{lem}

\begin{proof}
To derive a contradiction, we assume that there exists the boundary point $\zeta\in\partial\D$ such that $\beta_\phi(\zeta)\leq1$, but $\lim_{z\rightarrow\zeta}\frac{\omega(z)}{\omega(\phi(z))}=0$. Then, by (\ref{d<=1}) we obtain the prompt result $|\phi(r\zeta)|\geq r$ for $0<r<1$. Thus, $\lim_{r\rightarrow1}\frac{\omega(r)}{\omega(\phi(r\zeta))}\geq1$, so this completes the proof.
\end{proof}

After, we conclude that the condition $\beta_\phi(\zeta)>1$ is an equivalent condition for $\lim_{z\rightarrow\zeta}\frac{\omega(z)}{\omega(\phi(z))}=0$. We can find its proof in lemma 3.1 of \cite{P} for the other side.

\section{Sufficiency for compact differences}

\begin{lem}\label{easypart}
Let $\omega$ be a regular fast weight and $U$ be a nonnegative bounded measurable function on $\D$. If $C_\phi$ is bounded on $A^2(\omega)$ and
\begin{align*}
\lim_{|z|\rightarrow1}U(z)\frac{\omega(z)}{\omega(\phi(z))}=0,
\end{align*}
then $\omega^{-1}[U\omega dA]\circ\phi^{-1}$ is a vanishing Carleson measure on $A^p(\omega)$, $0<p<\infty$.
\end{lem}

\begin{proof}
For a given $\epsilon>0$, there exists $\delta_1>0$ such that
\begin{align}\label{assump}
U(z)\frac{\omega(z)}{\omega(\phi(z))}<\epsilon\quad\text{for}\quad1-|z|<\delta_1.
\end{align}
From the Schwartz-Pick theorem, we have $C\geq1$ such that $1-|z|\leq C(1-|\phi(z)|)$ for all $z\in\D$. Now, consider a point $z\in\D$ that satisfies
\begin{align}\label{condiz}
1-|z|<\frac{\delta_1}{C(c_1\delta+1)},\quad0<\delta<m_\tau,
\end{align}
where the constant $c_1$ appears at the very front of Section 2. For any $\xi\in\phi^{-1}(D(\delta\tau(z)))$, we have
\begin{align*}
c_1\delta(1-|z|)\geq\delta\tau(z)>|\phi(\xi)-z|\geq||z|-|\phi(\xi)||\geq(1-|\phi(\xi)|)-(1-|z|),
\end{align*}
so that we can obtain $1-|\phi(\xi)|\leq(c_1\delta+1)(1-|z|)$. Thus, by (\ref{condiz}) we have
\begin{align}\label{lemmakey}
1-|\xi|\leq C(1-|\phi(\xi)|)\leq C(c_1\delta+1)(1-|z|)<\delta_1
\end{align}
for any $\xi\in\phi^{-1}(D(\delta\tau(z)))$. On the other hand, the boundedness condition (\ref{suffibddwithoutU}) and Carleson measure theorem say that $C_\phi$ is also bounded by $A^p(\omega^\alpha)$ for all $\alpha\in\R^+$ and all $0<p<\infty$ whenever $C_\phi$ is bounded by $A^2(\omega)$. Thus, for $z\in\D$ with condition (\ref{condiz}), we have
\begin{align*}
\omega^{-1}[U\omega dA]\circ\phi^{-1}(D(\delta\tau(z)))&=\int_{\phi^{-1}(D(\delta\tau(z)))}U(\xi)\frac{\omega(\xi)}{\omega(\phi(\xi))}\,dA(\xi)\\
&\leq\epsilon^{\frac{1}{2}}\int_{\phi^{-1}(D(\delta\tau(z)))}\left[U(\xi)\frac{\omega(\xi)}{\omega(\phi(\xi))}\right]^{\frac{1}{2}}\,dA(\xi)\\
&\lesssim\epsilon^{\frac{1}{2}}\int_{\phi^{-1}(D(\delta\tau(z)))}\left[\frac{\omega(\xi)}{\omega(\phi(\xi))}\right]^{\frac{1}{2}}\,dA(\xi)\\
&\lesssim\epsilon^{\frac{1}{2}}\tau(z)^2.
\end{align*}
Here, the first inequality is from (\ref{assump}) and (\ref{lemmakey}), while the last inequality holds since $C_\phi$ is also bounded by $A^p(\omega^{1/2})$ and the sizes of the Carleson boxes of $A^p(\omega^{1/2})$ are comparable to those of $A^p(\omega)$. Thus, the proof is complete.
\end{proof}

 We note that if the bounded measurable function $U$ appearing in Lemma \ref{easypart} is replaced by $|U|^p$, (\ref{ucphicompact}) can be a sufficient condition for the compactness of the weighted composition operator $UC_\phi f=U(f\circ\phi)$ on $A^p(\omega)$ since $\omega^{-1}[|U|^p\omega dA]\circ\phi^{-1}$ is a vanishing Carleson measure on $A^p(\omega)$. If we give the analyticity on $U$, we can obtain the following equivalent relation.

\begin{cor}
Let $\omega$ be a regular fast weight and $U$ be a bounded analytic function on $\D$. If $C_\phi$ is bounded by $A^2(\omega)$, then the weighted composition operator $UC_\phi$ is compact on $A^p(\omega)$, $0<p<\infty$ if and only if
\begin{align}\label{ucphicompact}
\lim_{|z|\rightarrow1^-}|U(z)|^p\frac{\omega(z)}{\omega(\phi(z))}=0.
\end{align}
\end{cor}

\begin{proof}
As mentioned above, Lemma \ref{easypart} gives a sufficient condition for the compactness of $UC_\phi$ promptly. Conversely, we assume that there exists a sequence $\{z_n\}$ that tends to some boundary point, $\zeta$, such that
\begin{align*}
\lim_{n\rightarrow\infty}|U(z_n)|^p\frac{\omega(z_n)}{\omega(\phi(z_n))}\neq0.
\end{align*}
Then, using the test function $\{g_n\}$, which converges weakly to zero in $A^p(\omega)$ defined in (\ref{testft}), Lemma \ref{subharmoni-1} and Lemma \ref{delight} give the following inequality
\begin{align*}
\|UC_\phi g_n\|^p_p&\geq\int_{D(\delta\tau(z_n))}|U(\xi)g_n(\phi(\xi))|^p\omega(\xi)dA(\xi)\\
&\gtrsim|U(z_n)|^p\tau(z_n)^2\omega(z_n)|g_n(\phi(z_n))|^p\\
&\gtrsim|U(z_n)|^p\frac{\tau(z_n)^2}{\tau(\phi(z_n))^2}\omega(z_n)|G_{\phi(z_n),b}(\phi(z_n))|^p\\
&\gtrsim|U(z_n)|^p\frac{\omega(z_n)}{\omega(\phi(z_n))},
\end{align*}
which derives a contradiction to the compactness of $UC_\phi$, thus we completed the proof.
\end{proof}

From the assumption $\tau(z)\leq c_1(1-|z|)$ for $z\in\D$, we see that the distance $d_\tau$ is also complete since we have

\begin{align*}
d_\tau(z,w)\gtrsim d_h(z,w)
\end{align*}

where $d_h(z,w)$ is the hyperbolic distance associated with the Bergman metric.

\begin{lem}
Given the radius function $\kappa$, we can define the bounded function $\rho_\kappa$ by
\begin{align*}
\rho_\kappa(z,w):=1-e^{-d_\kappa(z,w)},\quad z,w\in\D.
\end{align*}
Then, $\rho_\kappa(z,w)$ is a distance in $\D$.
\end{lem}

\begin{proof}
It is obvious that $\rho_\kappa$ is positive and symmetric. Moreover, $\rho_\kappa=0$ whenever $z=w$ since $d_\kappa(z,w)$ is a distance in $\D$. Thus, it remains to be proved that $\rho_\kappa$ holds the triangle inequality. Let $f(x)=1-e^{-x}$, where $x\in[0,\infty)$. If we consider the following function
\begin{align*}
F(x)=f(x+h)-f(x)-f(h)=e^{-x}+e^{-h}-e^{-x-h}-1,
\end{align*}
we easily get $F(0)=0$ and $F'(x)=-e^{-x}+e^{-x-h}\leq0$ when $x,h\in[0,\infty)$. Thus, we have $f(x+h)\leq f(x)+f(h)$. Finally, since $d_\kappa(z,w)\leq d_\kappa(z,\xi)+d_\kappa(\xi,w)$ for $z,w,\xi\in\D$ and $f$ is an increasing function, we obtain
\begin{align*}
f(d_\kappa(z,w))\leq f(d_\kappa(z,\xi)+d_\kappa(\xi,w))\leq f(d_\kappa(z,\xi))+f(d_\kappa(\xi,w)).
\end{align*}
Therefore, we conclude that $\rho_\kappa=f\circ d_\kappa$ is a distance.
\end{proof}

\begin{thm}\label{mainresult2}
Let $\omega$ belong to the class $\mathcal{W}$. Suppose $C_\phi$, $C_\psi$ are bounded on $A^p(\omega)$ for some $p\in(0,\infty)$ and there is $R>0$ such that $d_\varphi(\phi(z),\psi(z))<R$ for all $z\in\D$. If
\begin{align*}
\lim_{|z|\rightarrow1^-}\rho_\tau(\phi(z),\psi(z))\left(\frac{\omega(z)}{\omega(\phi(z))}+\frac{\omega(z)}{\omega(\psi(z))}\right)=0,
\end{align*}
then $C_\phi-C_\psi$ is compact on $A^p(\omega)$ for all $p\in(0,\infty)$.
\end{thm}

\begin{proof}
Choose a small number $\epsilon>0$ with
\begin{align}\label{epsilon}
0<\epsilon<\frac{\delta}{4}\quad\text{where}\quad0<\delta<m_\tau,
\end{align}
and we define the set $E$ by
\begin{align*}
E:=\{z\in\D:d_\tau(\phi(z),\psi(z))<\epsilon\}.
\end{align*}
If $\{f_k\}$ is any bounded sequence in $A^p(\omega)$ such that $f_k$ uniformly converges to $0$ on compact subsets of $\D$ as $k\rightarrow\infty$, then
\begin{align}\label{goal}
&\|(C_\phi-C_\psi)f_k\|^p_p\nonumber\\
&=\int_{E}|f_k\circ\phi-f_k\circ\psi|^pe^{-\varphi}dA+\int_{E^c}|f_k\circ\phi-f_k\circ\psi|^pe^{-\varphi}dA.
\end{align}
First, we can make the first integral of (\ref{goal}) as small as desired. Since $|\phi(z)-\psi(z)|<2\epsilon\tau(\phi(z))$ for $z\in E$ by Proposition \ref{setinclus}, Lemma \ref{inequality} and Fubini theorem follow
\begin{align}\label{integralE1}
&\int_{E}|f_k\circ\phi(z)-f_k\circ\psi(z)|^pe^{-\varphi(z)}dA(z)\nonumber\\
&\lesssim\int_{E}\frac{\rho_\tau(\phi(z),\psi(z))^pe^{\varphi(\phi(z))-\varphi(z)}}{\tau(\phi(z))^2}
\int_{D(\delta\tau(\phi(z)))}|f_k(\xi)|^pe^{-\varphi(\xi)}dA(\xi)dA(z)\nonumber\\
&\lesssim(1-e^{-\epsilon})^p\int_{\D}|f_k(\xi)|^pe^{-\varphi(\xi)}\int_{\phi^{-1}(D(\delta\tau(\xi)))}
\frac{1}{\tau(\phi(z))^2}\frac{e^{-\varphi(z)}}{e^{-\varphi(\phi(z))}}dA(z)dA(\xi)\nonumber\\
&\lesssim(1-e^{-\epsilon})^p\sup_{\xi\in\D}\frac{e^{\varphi}[e^{-\varphi}dA]\circ\phi^{-1}(D(\delta\tau(\xi)))}{\tau(\xi)^2}\|f_k\|_p^p
\lesssim(1-e^{-\epsilon})^p.
\end{align}
The last inequality comes from (\ref{opnorm}) and the assumption of the boundedness of $C_\phi$. Therefore, the first integral of (\ref{goal}) also can be dominated by an arbitrary small number. To calculate the second integral part of (\ref{goal}), we recall that our assumption gives
\begin{align*}
\lim_{|z|\rightarrow1}\chi_{E^c}(z)\rho_\tau(\phi(z),\psi(z))\left(\frac{\omega(z)}{\omega(\phi(z))}+
\frac{\omega(z)}{\omega(\psi(z))}\right)=0.
\end{align*}
By Lemma \ref{easypart}, we conclude that $\omega^{-1}[\chi_{E^c}\omega dA]\circ\phi^{-1}$ and $\omega^{-1}[\chi_{E^c}\omega dA]\circ\psi^{-1}$ are vanishing Carleson measures on $A^p(\omega)$, thus the second term of (\ref{goal}) is
\begin{align}\label{E^c}
&\int_{E^c}|f_k\circ\phi-f_k\circ\psi|^pe^{-\varphi} dA\\
&\leq\int_{\D}|f_k|^pe^{-\varphi}e^{\varphi}[\chi_{E^c}e^{-\varphi}dA]\circ\phi^{-1}+\int_{\D}|f_k|^pe^{-\varphi}e^{\varphi}[\chi_{E^c}e^{-\varphi}dA]
\circ\psi^{-1}\rightarrow0\nonumber
\end{align}
when $k\rightarrow\infty$. Hence, we completed the proof.
\end{proof}

As a consequence of Theorem \ref{mainresult2}, we obtained a sufficient condition whereby a single composition operator can be represented by a finite sum of the composition operators modulo the compact operator.
\begin{lem}\label{delight}
Let $\omega\in\mathcal{W}$. If there is a curve $\gamma$ connecting to $\zeta\in\partial\D$ and a constant $c>0$ such that $\lim_{z\to \zeta}\frac{\omega(z)}{\omega(\phi(z))}\geq c$ where $z\in\gamma$, then
\begin{align*}
\liminf_{z\to \zeta}\frac{\tau(z)}{\tau(\phi(z))}\geq\min(1,c),\quad\text{for}\quad z\in\gamma.
\end{align*}
Moreover, if $\lim_{z\to \zeta}\frac{\tau(z)}{\tau(\phi(z))}=0$, then $\lim_{z\to \zeta}\frac{\omega(z)}{\omega(\phi(z))}=0$.
\end{lem}

\begin{proof}
For points $|\phi(z)|>|z|$ for $z\in\gamma$, we easily obtain $\frac{\tau(z)}{\tau(\phi(z))}\geq1$. On the other hand, from (\ref{lowerbd}) of Lemma, we have
\begin{align*}
\frac{d}{dr}\frac{\omega(r)}{\tau(r)}=\frac{-\varphi'(r)e^{-\varphi(r)}\tau(r)-e^{-\varphi(r)}\tau'(r)}{\tau(r)^2}\leq-\frac{e^{-\varphi(r)}}{\tau(r)^2}\left(1+\tau'(r)\right)<0
\end{align*}
near the boundary and the assumption $\tau'(r)\rightarrow0$ as $r\rightarrow1^-$. Thus, for points $|\phi(z)|\leq|z|$ on the curve $\gamma$, we have
\begin{align}\label{w/t}
\frac{\tau(z)}{\tau(\phi(z))}\geq\frac{\omega(z)}{\omega(\phi(z))}.
\end{align}
Therefore, $\frac{\tau(z)}{\tau(\phi(z))}\geq\frac{\omega(z)}{\omega(\phi(z))}\geq c$, so we have proved the first statement. For the second statement, since $\lim_{z\to \zeta}\frac{\tau(z)}{\tau(\phi(z))}=0$, we have $|\phi(z)|\leq|z|$ near the point $\zeta$. Thus, we easily get $\lim_{z\to \zeta}\frac{\omega(z)}{\omega(\phi(z))}=0$ by (\ref{w/t}).
\end{proof}

Denote the subset of boundary points by the following set:

\begin{align}\label{Fset}
F(\phi):=\left\{\zeta\in\partial\D:\limsup_{z\to \zeta}\frac{\tau(z)}{\tau(\phi(z))}>0\right\}.
\end{align}
We can see that $C_\phi$ is compact on $A^p(\omega)$ if $F(\phi)=\emptyset$ by Lemma \ref{delight}.
\begin{thm}
Let $\omega\in\mathcal{W}$ and $C_\phi, C_{\phi_1},\ldots, C_{\phi_M}$ be bounded on $A^p(\omega)$ for some $p\in(0,\infty)$. Suppose $F(\phi_i)\cap F(\phi_j)=\emptyset$ for $i\neq j$ and $F(\phi)=\bigcup^M_{j=1}F(\phi_j)$. If for each $j=1,\ldots,M$, $d_\varphi(\phi(z),\phi_j(z))<R$ for all $z\in\D$ and some $R>0$ and
\begin{align*}
\lim_{z\rightarrow\zeta}\rho_\tau(\phi(z),\phi_j(z))\left(\frac{\omega(z)}{\omega(\phi(z))}+\frac{\omega(z)}{\omega(\phi_j(z))}\right)=0,\quad\forall\zeta\in F(\phi_j),
\end{align*}
then, $C_\phi-\sum^M_{j=1}C_{\phi_j}$ is compact on $A^p(\omega)$.
\end{thm}

\begin{proof}
We define the subsets $D_i$ of $\D$ by
\begin{align*}
D_i:=\{z\in\D:|\phi_i(z)|\geq|\phi_j(z)|,\quad \forall j\neq i\},\quad i=1,\ldots,M,
\end{align*}
and their subsets $E_i:=\{z\in D_i:d_\tau(\phi(z),\phi_i(z))<\epsilon\}$, where $0<\epsilon<\frac{\delta}{4}$. Now, let $\{f_k\}$ be a bounded sequence which converges to zero weakly when $k\rightarrow\infty$. Since $\D=\bigcup^M_{i=1}D_i$, we have
\begin{align}\label{goal1}
\|(C_\phi-\sum^M_{j=1}C_{\phi_j})f_k\|_{p}^p&\leq\sum^M_{i=1}\int_{E_i}|f_k\circ\phi-f_k\circ\phi_1-\ldots-f_k\circ\phi_M|^pe^{-\varphi}dA\nonumber\\
&+\sum^M_{i=1}\int_{D_i\setminus E_i}|f_k\circ\phi-f_k\circ\phi_1-\ldots-f_k\circ\phi_M|^pe^{-\varphi}dA.
\end{align}
 Since our assumptions say that
\begin{align*}
0=\lim_{|z|\rightarrow1}\chi_{D_i\setminus E_i}(z)\rho_\tau(\phi(z),\phi_i(z))\frac{\omega(z)}{\omega(\phi_i(z))}\geq\lim_{|z|\rightarrow1}\epsilon\chi_{D_i\setminus E_i}(z)\frac{\omega(z)}{\omega(\phi_j(z))},
\end{align*}
and $\lim_{|z|\rightarrow1}\chi_{D_i\setminus E_i}(z)\frac{\omega(z)}{\omega(\phi(z))}=0$, the second integral of (\ref{goal1}) vanishes as $k\rightarrow\infty$ by the triangle inequality and Lemma \ref{easypart}. For the first integral of (\ref{goal1}), we have
\begin{align}\label{goal2}
&\int_{E_i}|f_k\circ\phi-f_k\circ\phi_1-\ldots-f_k\circ\phi_M|^pe^{-\varphi}dA\nonumber\\&\lesssim\int_{\D}|f_k\circ\phi-f_k\circ\phi_i|^p\chi_{E_i}e^{-\varphi}dA
+\sum^M_{j\neq i}\int_{\D}|f_k\circ\phi_j|^p\chi_{E_i}e^{-\varphi}dA.
\end{align}
Since $F(\phi_i)\cap F(\phi_j)=\emptyset$ and $\frac{\omega(z)}{\omega(\phi_i(z))}\geq\frac{\omega(z)}{\omega(\phi_j(z))}$ for $z\in D_i$ when $j\neq i$, we have
\begin{align*}
0=\lim_{z\rightarrow\zeta}\chi_{E_i}(z)\frac{\omega(z)}{\omega(\phi_i(z))}\geq\lim_{z\rightarrow\zeta}\chi_{E_i}(z)\frac{\omega(z)}{\omega(\phi_j(z))}\quad\text{for}\quad \zeta\in F(\phi_j),
\end{align*}
by Lemma \ref{delight}. Therefore $\lim_{|z|\rightarrow1^-}\chi_{E_i}(z)\frac{\omega(z)}{\omega(\phi_j(z))}=0$ when $j\neq i$, thus the second term of (\ref{goal2}) vanishes as $k\rightarrow\infty$ by Lemma \ref{easypart}. Finally, we must still prove that the first integral of (\ref{goal2}) converges to $0$ when $k\rightarrow\infty$. To prove this, it is enough to show that the following statement holds owing to Theorem \ref{mainresult2},
\begin{align*}
\lim_{|z|\rightarrow1}\chi_{E_i}(z)\rho_\tau(\phi(z),\phi_i(z))\left(\frac{\omega(z)}{\omega(\phi(z))}+\frac{\omega(z)}{\omega(\phi_i(z))}\right)=0,\quad i=1,\ldots,M.
\end{align*}
Suppose that there is a boundary point $\zeta\notin F(\phi_i)$ and a sequence $\{z_n\}$ that converges to $\zeta$ such that $\frac{\omega(z_n)}{\omega(\phi(z_n))}\geq\frac{\omega(z_n)}{\omega(\phi_i(z_n))}$ and
\begin{align*}
\lim_{n\rightarrow\infty}\chi_{E_i}(z_n)\rho_\tau(\phi(z_n),\phi_i(z_n))\frac{\omega(z_n)}{\omega(\phi(z_n))}\neq0.
\end{align*}
Then, $\rho_\tau(\phi(z_n),\phi_i(z_n))\geq d_1$ and $\frac{\omega(z_n)}{\omega(\phi(z_n))}\geq d_2$ for some positive numbers $d_1, d_2>0$. Since $|\phi(z_n)|\geq|\phi_i(z_n)|$, for any arbitrary curve $\gamma$ with $\gamma(0)=\phi(z_n)$ and $\gamma(1)=\phi_i(z_n)$, we have
\begin{align*}
\int^1_0\frac{|\gamma'(t)|}{\tau(\gamma(t))}dt&\geq\int^1_0\frac{-\tau'(\gamma(t))}{\tau(\gamma(t))}|\gamma'(t)|dt\\&\geq\left|\int^1_0\frac{-\tau'(\gamma(t))}{\tau(\gamma(t))}\gamma'(t)dt\right|
=\left|\log\tau(\phi_i(z_n))-\log\tau(\phi(z_n))\right|.
\end{align*}
 Thus, $d_\tau(\phi(z_n),\phi_i(z_n))\geq\log\frac{\tau(\phi_i(z_n))}{\tau(\phi(z_n))}$ and
\begin{align*}
\rho_\tau(\phi(z_n),\phi_i(z_n))=1-e^{-d_\tau(\phi(z_n),\phi_i(z_n))}\geq1-\frac{\tau(\phi(z_n))}{\tau(\phi_i(z_n))}.
\end{align*}
But, by Lemma \ref{delight}, we obtain $\frac{\tau(\phi(z_n))}{\tau(\phi_i(z_n))}=\frac{\tau(z_n)}{\tau(\phi_i(z_n))}/\frac{\tau(z_n)}{\tau(\phi(z_n))}\rightarrow0$ as $n\rightarrow\infty$ since $\zeta\notin F(\phi_i)$ and $\frac{\omega(z_n)}{\omega(\phi(z_n))}\geq d_2$, thus
\begin{align*}
\rho_\tau(\phi(z_n),\phi_i(z_n))\rightarrow1,
\end{align*}
which is a contradiction because $\rho_\tau(\phi(z),\phi_i(z))\leq1-e^{-\epsilon}<1$ for $z\in E_i$. Thus, we completed our proof.
\end{proof}

\section{Necessity for compact differences}

In this section, we give necessary conditions for $C_\phi$ to be expressed as a single or a finite sum of the composition operators modulo the compacts.

\begin{lem}\label{necessary}
If $d_\tau(z,w)\geq2\delta$ for $0<\delta<m_\tau$, then $|z-w|\geq\delta\tau(z)$.
\end{lem}

\begin{proof}
Suppose there are two points $z,w$ such that $d_\tau(z,w)\geq2\delta$, but $|z-w|<\delta\tau(z)$. If we consider $\gamma(t)=(1-t)z+tw$, then by (\ref{equiquan})
\begin{align*}
2\delta\leq d_\tau(z,w)\leq\int^{1}_0\frac{|\gamma'(t)|}{\tau(\gamma(t))}dt\leq2\frac{|z-w|}{\tau(z)}<2\delta.
\end{align*}
Thus, it is a desired contradiction.
\end{proof}

\begin{lem}\cite{PP}\label{peakft}
Let $\omega\in\mathcal{W}$ and $0<p<\infty$, $b\in \mathbb N$ with $bp\geq1$. Given $R>0$ and $a\in\D$, where $|a|$ is sufficiently close to $1$, there exists a function $G_{a,b}$ analytic in $\D$ such that
\begin{align*}
|G_{a,b}(z)|^pe^{-\varphi(z)}\approx1\quad\text{if}\quad|z-a|<R\tau(a),
\end{align*}
\begin{align*}
|G_{a,b}(z)|^pe^{-\varphi(z)}\leq C(\varphi,R)\min\left[1,\frac{\min(R\tau(z),R\tau(a))}{|z-a|}\right]^{3bp},
\end{align*}
and
\begin{align*}
\|G_{a,b}\|^p_p\approx\tau(a)^2.
\end{align*}
\end{lem}

\begin{thm}\label{mainresult1}
Let $\omega\in\mathcal{W}$ and $C_\phi, C_{\psi}$ be bounded on $A^p(\omega)$ for some $p\in(0,\infty)$. If $C_\phi-C_\psi$ is compact on $A^p(\omega)$ then
\begin{align*}
\lim_{|z|\rightarrow1}\rho_\tau(\phi(z),\psi(z))\left(\frac{\omega(z)}{\omega(\phi(z))}+\frac{\omega(z)}{\omega(\psi(z))}\right)=0.
\end{align*}
\end{thm}

\begin{proof}
Let us assume that there is the boundary point $\zeta$ and the sequence $\{z_n\}$ such that
\begin{align*}
\lim_{z_n\rightarrow\zeta}\rho_\tau(\phi(z_n),\psi(z_n))\frac{\omega(z_n)}{\omega(\phi(z_n))}\neq0.
\end{align*}
Then, there are constants $d_1, d_2>0$ and a large number $N>0$ that satisfy
\begin{align}\label{distcondi}
d_\tau(\phi(z_n),\psi(z_n))\geq d_1,
\end{align}
and $\frac{\omega(z_n)}{\omega(\phi(z_n))}\geq d_2$ with $|\phi(z_n)|\geq|\psi(z_n)|$ for $n>N$. Now, we take the following sequence $\{g_n\}$ introduced in Lemma \ref{peakft}:
\begin{align}\label{testft}
g_n(\xi)=\frac{G_{\phi(z_n),b}(\xi)}{\tau(\phi(z_n))^{2/p}},
\end{align}
which converges weakly to $0$ and satisfies
\begin{align}\label{testcondi}
|G_{\phi(z_n),b}(\xi)|^pe^{-\varphi(z)}\lesssim\left(\frac{\min(\frac{d_1}{4}\tau(\xi),\frac{d_1}{4}\tau(\phi(z_n)))}{|\xi-\phi(z_n)|}\right)^{3bp}
\end{align}
for $|\xi-\phi(z_n)|\geq \frac{d_1}{4}\tau(\phi(z_n))$. We will choose the proper number $b$ later. Applying Lemma \ref{subharmoni-1} and (\ref{testft}), we have
\begin{align}\label{goal3}
\|(C_\phi-C_{\psi})g_{n}\|_p^p&\geq\int_{D(\delta\tau(z_n))}|g_{n}(\phi(\xi))-g_{n}(\psi(\xi))|^p\omega(\xi)dA(\xi)\nonumber\\
&\geq\tau(z_n)^2|g_n(\phi(z_n))-g_{n}(\psi(z_n))|^p\omega(z_n)\nonumber\\
&\gtrsim\frac{\tau(z_n)^2}{\tau(\phi(z_n))^2}|G_{\phi(z_n),b}(\phi(z_n))-G_{\phi(z_n),b}(\psi(z_n))|^p\omega(z_n)\nonumber\\
&\gtrsim\frac{\tau(z_n)^2}{\tau(\phi(z_n))^2}\frac{\omega(z_n)}{\omega(\phi(z_n))}\left|1-\left|
\frac{G_{\phi(z_n),b}(\psi(z_n))}{G_{\phi(z_n),b}(\phi(z_n))}\right|\right|^p.
\end{align}
Here, Lemma \ref{peakft} and (\ref{testcondi}) state that there is a constant $C(d_1)>0$ such that
\begin{align}\label{difficult}
\left|\frac{G_{\phi(z_n),b}(\psi(z_n))}{G_{\phi(z_n),b}(\phi(z_n))}\right|^p\leq C\frac{\omega(\phi(z_n))}{\omega(\psi(z_n))}\min\left[1,\frac{\min(\frac{d_1}{4}\tau(\phi(z_n)),\frac{d_1}{4}\tau(\psi(z_n)))}
{|\phi(z_n)-\psi(z_n)|}\right]^{6bp}.
\end{align}
Lemma \ref{necessary} and (\ref{distcondi}) enable the right side of (\ref{difficult}) to be less than $1$ by taking a sufficiently large $b$ such that $bp\geq1$. By our assumption $\frac{\omega(z_n)}{\omega(\phi(z_n))}\geq d_2$, Lemma \ref{delight} and $|\phi(z_n)|\geq|\psi(z_n)|$, we derive a contradiction since (\ref{goal3}) does not vanish when $n\rightarrow\infty$. Thus, we completed our proof.
\end{proof}

As a consequence of Theorem \ref{mainresult1}, we will characterize the case where a single composition operator is represented by modulo compact operators to a sum of finitely many composition operators.

\begin{thm}\label{necesum}
Let  $\omega\in\mathcal{W}$ and $C_\phi, C_{\phi_1},\ldots, C_{\phi_M}$ be bounded operators on $A^p(\omega)$ for some $p\in(0,\infty)$. We define a set as
\begin{align*}
F(\phi)=\left\{\zeta\in\partial\D:\limsup_{z\to \zeta}\frac{\tau(z)}{\tau(\phi(z))}>0\right\}.
\end{align*}
Suppose $F(\phi_i)\cap F(\phi_j)=\emptyset$ for $i\neq j$ and $F(\phi)=\bigcup^M_{j=1}F(\phi_j)$. If $C_\phi-\sum^M_{j=1}C_{\phi_j}$ is compact on $A^p(\omega)$ then for each $j=1,\ldots,M$:
\begin{align*}
\lim_{z\rightarrow\zeta}\rho_\tau(\phi(z),\phi_j(z))\left(\frac{\omega(z)}{\omega(\phi(z))}+\frac{\omega(z)}{\omega(\phi_j(z))}\right)=0,\quad\forall\zeta\in F(\phi_j).
\end{align*}
\end{thm}

\begin{proof}
Assume that there is the boundary point $\zeta\in F(\phi_i)$ and the sequence $\{z_n\}$ on $\D$ such that
\begin{align*}
\lim_{z_n\rightarrow\zeta}\rho_\tau(\phi(z_n),\phi_i(z_n))\frac{\omega(z_n)}{\omega(\phi(z_n))}\neq0
\end{align*}
so that $d_\tau(\phi(z_n),\phi_i(z_n))\geq d_1$ and $\frac{\omega(z_n)}{\omega(\phi(z_n))}\geq d_2$ with $|\phi(z_n)|\geq|\phi_i(z_n)|$ for some $d_1, d_2>0$. Now, the same argument as in the proof of Theorem \ref{mainresult1} and the same test functions $\{g_n\}$ in (\ref{testft}) yield
\begin{align*}
&\big\|(C_\phi-\sum^M_{j=1}C_{\phi_j})g_{n}\big\|^p_p\\
&\gtrsim\frac{\tau(z_n)^2}{\tau(\phi(z_n))^2}|G_{\phi(z_n),b}(\phi(z_n))-\sum^M_{j=1}G_{\phi(z_n),b}(\phi_j(z_n))|^p\omega(z_n)\\
&\gtrsim\frac{\tau(z_n)^2}{\tau(\phi(z_n))^2}\frac{\omega(z_n)}{\omega(\phi(z_n))}\left|1-\left|\frac{G_{\phi(z_n),b}(\phi_i(z_n))}{G_{\phi(z_n),b}(\phi(z_n))}\right|-
\sum^M_{j\neq i}\left|\frac{G_{\phi(z_n),b}(\phi_j(z_n))}{G_{\phi(z_n),b}(\phi(z_n))}\right|\right|^p.
\end{align*}
By Lemma \ref{peakft} and the condition (\ref{testcondi}) we have
\begin{align}\label{hope}
\left|\frac{G_{\phi(z_n),b}(\phi_j(z_n))}{G_{\phi(z_n),b}(\phi(z_n))}\right|^p
\lesssim\frac{\omega(\phi(z_n))}{\omega(\phi_j(z_n))}\min\left[1,\frac{\min(\frac{d_1}{4}\tau(\phi(z_n)),\frac{d_1}{4}\tau(\phi_j(z_n)))}
{|\phi(z_n)-\phi_j(z_n)|}\right]^{3bp}
\end{align}
for each $j=1,\ldots,M$. Since our assumption states that $\zeta\in F(\phi)$ but $\zeta\notin F(\phi_j)$ when $j\neq i$, by Lemma \ref{delight} we have
\begin{align*}
\frac{\omega(\phi(z_n))}{\omega(\phi_j(z_n))}\leq\frac{1}{d_2}\frac{\omega(z_n)}{\omega(\phi_j(z_n))}\longrightarrow0\quad\text{as}\quad n\rightarrow\infty.
\end{align*}
Therefore, we see that the right side of the inequality (\ref{hope}) vanishes as $n\rightarrow\infty$. When $j=i$, we can take the test function $G_{\phi(z_n),b}$ with large enough $b$ to make the right side of (\ref{hope}) small since $\frac{\omega(\phi(z_n))}{\omega(\phi_i(z_n))}\leq1$. Thus, we completed our proof.
\end{proof}

\section{Applications}
As an interesting consequence of Theorem \ref{mainresult2}, we obtain a sufficient condition for which two composition operators belong to the same path component of the bounded composition operator space $\mathcal{C}(A^p(\omega))$.

\begin{thm}\label{partialconj}
Let $\omega\in\mathcal{W}$ and $\mathcal{C}(A^p(\omega))$ be a space of bounded composition operators on $A^p(\omega)$, $0<p<\infty$. If there exists $R>0$ such that $d_\varphi(\phi(z),\psi(z))<R$ for all $z\in\D$, then $C_\phi$ and $C_\psi$ are path connected in $\mathcal{C}(A^p(\omega))$.
\end{thm}

\begin{proof}
Suppose $C_\phi$ and $C_\psi$ are bounded operators on $A^p(\omega)$. By Proposition \ref{impot} and (\ref{suffibddwithoutU}) we can tell that the composition operator $C_{\phi_s}$ is also bounded for all $s\in[0,1]$. Now, we show that there exists a continuous path $\mathbf{t}:[0,1]\mapsto C_{\phi_t}$ in $\mathcal{C}(A^p(\omega))$ where $\phi_t=(1-t)\phi+t\psi$, i.e.
\begin{align*}
\lim_{t\rightarrow s}\|C_{\phi_t}-C_{\phi_s}\|\rightarrow0.
\end{align*}
By the mean value theorem, the inequality (\ref{integralE1}) and (\ref{opnorm}), we have
\begin{align}\label{goal4}
&\int_{\D}|f\circ\phi_t-f\circ\phi_s|^pe^{-\varphi}dA\nonumber\\
&\lesssim\int_{\D}|f'(\phi_u)||\phi_t-\phi_s|^pe^{-\varphi}dA\nonumber\\
&\lesssim\int_{\D}\frac{|\phi_t(z)-\phi_s(z)|^p}{\tau(\phi_u(z))^{2+p}}e^{\varphi(\phi_u(z))-\varphi(z)}\int_{D(\delta\tau(\phi_u(z)))}|f(\xi)|^pe^{-\varphi(\xi)}dA(\xi)dA(z)\nonumber\\
&\lesssim\sup_{z\in\D}\frac{|\phi_t(z)-\phi_s(z)|^p}{\tau(\phi_u(z))^{p}}\int_{\D}|f(\xi)|^p\frac{e^{-\varphi(\xi)}}{\tau(\xi)^2}\int_{\phi_u^{-1}(D(\delta\tau(\xi)))}
\frac{e^{-\varphi(z)}}{e^{-\varphi(\phi_u(z))}}dA(z)dA(\xi)\nonumber\\
&\lesssim\sup_{z\in\D}\frac{|\phi_t(z)-\phi_s(z)|^p}{\tau(\phi_u(z))^{p}}\|C_{\phi_u}\|.
\end{align}
 Therefore, using the fact $|\phi_u(z)|\leq\max(|\phi(z)|,|\psi(z)|)$ and (\ref{need}) of Lemma \ref{setinclus2}, there is $R'>0$ such that
\begin{align*}
\frac{|\phi_t(z)-\phi_s(z)|}{\tau(\phi_u(z))}&\lesssim|\phi_t(z)-\phi_s(z)|\varphi'(|\phi_u(z)|) \\&\lesssim|(s-t)(\phi(z)-\psi(z))|\max(\varphi'(|\phi(z)|),\varphi'(|\psi(z)|))\\&\lesssim R'|s-t|\longrightarrow0
\end{align*}
as $t\rightarrow s$. Thus, (\ref{goal4}) converges to $0$ as $t\rightarrow s$ since $\|C_{\phi_u}\|$ is uniformly bounded.
\end{proof}

\bibliographystyle{amsplain}

\end{document}